\documentclass[11pt]{amsart}
\usepackage{ amsmath}
\usepackage{amssymb}
\usepackage{amsxtra}
\usepackage{enumerate}
\usepackage[mathscr]{eucal}
\usepackage[margin=3.0 cm]{geometry}

\def\SL{{\rm SL}}
\def\C{\mathbb C}

\def\SL{{\rm SL}}

\def \qh{{\bf H}_{\mathbb H}}

\def\Sp{{\rm Sp}}

\newtheorem{theorem}{Theorem}[section]

\theoremstyle{definition}

\theoremstyle{remark}

\numberwithin{equation}{section}
\theoremstyle{plain}

\newtheorem{cor}[theorem]{Corollary}

\newtheorem{question}{Question}

\newtheorem*{theorema}{Theorem~{\rm WJC}}

\newcommand{\thmref}[1]{Theorem~\ref{#1}}

\newcommand{\eqnref}[1]{~{\textrm(\ref{#1})}}

\begin{document}
\title[On generalized J\o{}rgensen inequality in $\SL(2, \C)$ ]{On generalized J\o{}rgensen inequality in $\SL(2, \C)$}
\author[K.  Gongopadhyay]{Krishnendu Gongopadhyay}
 \address{ Indian Institute of Science Education and Research (IISER) Mohali,
Knowledge City, Sector 81, SAS Nagar, Punjab 140306, India}
\email{krishnendu@iisermohali.ac.in, krishnendug@gmail.com}
\author[M. M. Mishra]{Mukund Madhav Mishra}
\address{
  Department of Mathematics, Hansraj College, University of Delhi, Delhi 110007, India}
\email{mukund.math@gmail.com}
\author[D. Tiwari]{Devendra Tiwari}
\address{
  Department of Mathematics, University of Delhi, Delhi 110007, India}
\email{devendra9.dev@gmail.com}
\date{\today}
\thanks{Gongopadhyay acknowledges partial support from SERB MATRICS grant MTR/2017/000355.  Tiwari is supported by NBHM-SRF}

\begin{abstract}
Wang, Jiang and Cao have obtained a generalized version of the J\o{}rgensen inequality in { Proc. Indian Acad. Sci. Math. Sci., 123(2):245--251, 2013}, for two generator subgroups of $\SL(2, \C)$ where one of the generators is loxodromic. We prove that their inequality is strict. 
\end{abstract}
\maketitle

\section{Introduction}
The J\o{}rgensen inequality is a classical result that provides necessary condition of discreteness for a two generator subgroup of $\SL(2, \C)$, see \cite{j}. Extremality of the J\o{}rgensen inequality has been investigated in \cite{jk}. In the literature there are several generalizations of the J\o{}rgensen inequality, and extremalities of some of those inequalities have also been investigated, e.g. \cite{g1}, \cite{gm2}, \cite{m}, \cite{rv}. In this note we investigate extremality of one such generalized J\o{}rgensen inequality.

In \cite{wjc}, Wang, Jiang and Cao have obtained a generalized version of the J\o{}rgensen inequality for two generator subgroups of $\SL(2, \C)$ where one of the generators is loxodromic. We recall their result.

\begin{theorema} \label{th0} \cite{wjc} Let $g$, $h$ are elements in $\SL(2, \C)$ such that $g$ is loxodromic. Suppose that $g, h$ are of the form: for $|\lambda|>1$,
\begin{equation}\label{eqff} g=\begin{pmatrix} \lambda & 0 \\ 0 & \lambda^{-1} \end{pmatrix},  ~~~~~~~~ h=\begin{pmatrix} a & b \\ c & d \end{pmatrix}. \end{equation} 
Let $g$ be such that $M_g<1$,
where 
$$M_g=|\lambda-1| + |\lambda^{-1}-1|.$$
If $\langle g, h \rangle$ is discrete and non-elementary, then
\begin{equation} \label{eq1}  |abcd|^{\frac{1}{2}}\geq \frac{1-M_g}{M_g^2}. \end{equation}
\end{theorema}
We prove that the above inequality is strict, see \thmref{th1} below. Further, we note down a few generalized J\o{}rgensen type inequalities which are also strict. 
\section{Proof of Strictness of the Inequality} 
\begin{theorem} \label{th1}
Under the hypothesis of the above theorem, equality does not hold in {\rm \eqnref{eq1}}. 
\end{theorem}
\begin{proof} 
If possible suppose equality holds in \eqnref{eq1}. Let $h_1=hgh^{-1}$. Then
\begin{eqnarray*}
h_1&=& {\begin{pmatrix} a & b \\ c & d \end{pmatrix}} \begin{pmatrix}  \lambda & 0 \\ 0 & \lambda^{-1} \end{pmatrix}  \begin{pmatrix} d & - b \\ -c & a \end{pmatrix}\\
&=& \begin{pmatrix} ad \lambda -bc \lambda^{-1} & -(\lambda - \lambda^{-1})ab \\ 
(\lambda-\lambda^{-1} ) cd &  ad\lambda^{-1} -bc \lambda \end{pmatrix}\\
&=& \begin{pmatrix} a_1 & b_1 \\ c_1 & d_1 \end{pmatrix}.
\end{eqnarray*}
Note that
$$b_1 c_1=-(\lambda-\lambda^{-1})^2 abcd=-(\lambda-\lambda^{-1})^2 bc(1+bc). $$
\begin{eqnarray*}
|a_1d_1| &=& |1+(\lambda-\lambda^{-1})^2abcd|\\
& \leq & 1+|\lambda-1+1-\lambda^{-1}||abcd|\\
& \leq & 1+\frac{(1-M_g)^2}{M_g^2}\\
&\leq & \frac{(M_g+1-M_g)^2}{M_g^2}.\\
\end{eqnarray*}
Thus we have 
\begin{equation} \label{11} |a_1d_1|^{\frac{1}{2}} \leq \frac{1}{M_g}. \end{equation} 
Also, we have $|b_1| \leq M_g|ab|$, $|c_1|\leq M_g|cd|$. Hence
\begin{equation}\label{12} |b_1 c_1|^{\frac{1}{2} }\leq \frac{1-M_g}{M_g}.\end{equation} 
In particular, 
\begin{equation}\label{3} M_g(1+|b_1 c_1|^{\frac{1}{2} })<1.\end{equation} 
 Now note that $\langle g, h_1 \rangle$ is discrete and non-elementary. Discreteness of $\langle g,h_1\rangle$ is obvious, and if it was elementary, that would have implied that $g$ and $h$ had a common fixed point and hence, would have contradicted the assumption that $\langle g,h\rangle$ is non-elementary.
  So, applying Theorem~{\rm WJC} to $\langle g, h_1\rangle$, we have 
\begin{eqnarray*} 
\frac{1-M_g}{M_g^2}& \leq & |a_1 b_1 c_1 d_1|^{\frac{1}{2}} \leq |a_1d_1|^{\frac{1}{2} }|b_1 c_1|^{\frac{1}{2}}\\
& \leq & \frac{1-M_g}{M_g^2},  ~~\hbox{ by \eqnref{11} and \eqnref{12}}. \end{eqnarray*}
This implies, $$|a_1 b_1 c_1 d_1|^{\frac{1}{2}}=\frac{1-M_g}{M_g^2}.$$
Next we observe that
\begin{eqnarray*}
\bigg( \frac{1-M_g}{M_g^2}\bigg)^2 & \leq |a_1 b_1 c_1 d_1| \\
& \leq & |1+b_1c_1||b_1 c_1| \\
& \leq & |1+|b_1 c_1||(\lambda-\lambda^{-1})^2 |b_0c_0||a_0d_0|\\
& \leq & (\lambda-\lambda^{-1})^2(1+|b_1 c_1|) \bigg( \frac{1-M_g}{M_g^2}\bigg)^2 \\
&\leq & M_g^2 (1 +|b_1c_1|+2|b_1c_1|^{\frac{1}{2}}) \bigg( \frac{1-M_g}{M_g^2}\bigg)^2\\
& \leq & \big(M_g(1+|b_1c_1|^{\frac{1}{2}} \big)^2 \bigg( \frac{1-M_g}{M_g^2} \bigg)^2\\
& \leq & \bigg( \frac{1-M_g}{M_g^2}\bigg)^2, ~~\hbox{ by \eqnref{3}}. 
\end{eqnarray*}
Hence we have 
\begin{equation} \label{4} (\lambda-\lambda^{-1})^2(1+|b_1 c_1|)=1. \end{equation}

Noting  that 
$$tr^2(g)-4=(\lambda-\lambda^{-1})^2, ~~\hbox{ and } ~~tr[g, h_1]-2=-(\lambda-\lambda^{-1})^2 b_1 c_1,$$ the above equality implies that $\langle g, h_1\rangle$ satisfies equality in the classical J\o{}rgensen inequality. By a theorem of J\o{}rgensen and Kiikka, see \cite[Theorem 2]{jk}, this implies that $g$ is either elliptic or parabolic, which is a contradiction.  Hence equality can not hold in \eqnref{eq1}. 
\end{proof} 

Combining Theorem~{\rm WJC} and \thmref{th1}, we can rephrase the generalized J\o{}rgensen inequality as follows. 
\begin{theorem}\label{th2}
Let $g$, $h$ are elements in $\SL(2, \C)$ such that $g$ is loxodromic. Suppose that $g, h$ are of the form {\rm \eqnref{eqff}}. Let $g$ be such that $M_g<1$. If 
\begin{equation} \label{eq00}  |abcd|^{\frac{1}{2}}\leq \frac{1-M_g}{M_g^2},  \end{equation}
then $\langle g, h \rangle$ is either elementary or non-discrete. 
\end{theorem}

\subsection{Some more inequalities} 
 The main idea in \cite{wjc} was to embed $\SL(2, \C)$ into the isometry group $\Sp(1,1)$ of the one dimensional quaternionic hyperbolic space $\qh^1$, and then use quaternionic J\o{}rgensen inequality of Cao and Parker, see \cite[Theorem 1.1]{cp}.  In view of the above theorem, following arguments as used in the proof of \cite[Corollary 1.2]{cp},  we note the following. 

\begin{cor}
Let $g$, $h$ are elements in $\SL(2, \C)$ such that $g$ is loxodromic. Suppose that $g, h$ are of the form \eqnref{eqff}. Let $g$ be such that $M_g<1$.  
If $\langle g, h \rangle$ is discrete and non-elementary, then each of the following strict inequalities holds. 
\begin{equation} \label{1} |bc|^{\frac{1}{2}} > \frac{1-M_g}{M_g}. \end{equation}

\begin{equation} \label{2}   |1+bc|^{\frac{1}{2}} > \frac{1-M_g}{M_g}. \end{equation}

\begin{equation} \label{4} |1+bc|+|bc| > \frac{2(1-M_g)}{M_g^2}. \end{equation}
\end{cor} 
\begin{proof} 
If possible, suppose $|bc|^{\frac{1}{2}} \leq  \frac{1-M_g}{M_g}$. Then
$$|ad|^{\frac{1}{2}} |bc|^{\frac{1}{2}} \leq (1+|bc|^{\frac{1}{2}}) |bc|^{\frac{1}{2}} \leq \frac{1-M_g}{M_g^2}.$$
Using \thmref{th2}, $\langle g, h \rangle$ is either discrete or non-elementary. 

If possible suppose  $|1+bc|^{\frac{1}{2}} \leq  \frac{1-M_g}{M_g}$. Then  it follows similarly as above noting that  $|bc|^{\frac{1}{2}} \leq 1+|ad|^{\frac{1}{2}}$ and $ad-bc=1$. 

Finally, if $|1+bc|+|bc|  \leq \frac{2(1-M_g)}{M_g^2}$, then
$$|ad|^{\frac{1}{2}} |bc|^{\frac{1}{2}} \leq \frac{1}{2}(|ad| + |bc|) \leq \frac{1-M_g}{M_g^2},$$
and the result follows from \thmref{th2}. 

This completes the proof.
\end{proof} 

In view of the results noted in this communication, the following question is natural to ask.
\begin{question}
What are sharp bounds for the inequalities \eqnref{eq1} and \eqnref{1} -- \eqnref{4}? 
\end{question}


\begin{thebibliography}{881}
\bibitem{cp}
Wensheng Cao and John~R. Parker.
\newblock J\o rgensen's inequality and collars in {$n$}-dimensional
  quaternionic hyperbolic space.
\newblock {\em Q. J. Math.}, 62(3):523--543, 2011.

\bibitem{g1} Krishnendu Gongopadhyay. 
\newblock On J\o{}rgensen inequality in infinite dimension
\newblock {\em New York J. Math. }, to appear. arXiv:1808.06756. 

\bibitem{gm2}
Krishnendu Gongopadhyay and  Abhishek Mukherjee. 
\newblock Extremality of quaternionic {J}\o rgensen inequality.
\newblock {\em Hiroshima Math. J.}, 47(2):113--137, 2017.

\bibitem{j}
Troels J\o{}rgensen.
\newblock On discrete groups of {M}\"obius transformations.
\newblock {\em Amer. J. Math.}, 98(3):739--749, 1976.

\bibitem{jk}
Troels J\o{}rgensen and Maire Kiikka.
\newblock Some extreme discrete groups.
\newblock {\em Ann. Acad. Sci. Fenn. Ser. A I Math.}, 1(2):245--248, 1975.

\bibitem{m} A. V. Masle\u{i}.
\newblock On the {G}ehring-{M}artin-{T}an numbers and {T}an numbers of
              the elementary subgroups of {$\rm{PSL}(2,\Bbb{C})$}.
\newblock {\em  Mat. Zametki}  102(2):  255--269, 2017.
\newblock  translation in
{\em Math. Notes} 102(1-2): 219--231, 2017. 
\bibitem{rv} Du\v{s}an Repov\v{s} and Andrei Vesnin, 
\newblock On {G}ehring-{M}artin-{T}an groups with an elliptic generator.
\newblock {\em Bull. Aust. Math. Soc.}  94(2): 326--336,  2016.


\bibitem{wjc}
Hua Wang, Yueping Jiang, and Wensheng Cao.
\newblock Notes on discrete subgroups of {M}\"obius transformations.
\newblock {\em Proc. Indian Acad. Sci. Math. Sci.}, 123(2):245--251, 2013.

\end{thebibliography}
\end{document}